\documentclass[psamsfonts,10]{amsart}

\usepackage{amsmath,amssymb,mathrsfs}

\usepackage[pdftex]{graphicx,color}
\usepackage[pdftex,colorlinks=false]{hyperref}
\usepackage{times}

\numberwithin{equation}{section}

\def\a{\alpha}
\def\b{\beta}

\newcommand{\n}{\nu}

\def\gm{\gamma}

\def\s{\sigma}

\def\C{\ensuremath{\mathbb{C}}}
\def\D{\ensuremath{\mathbb{D}}}

\def\R{\ensuremath{\mathbb{R}}}

\def\Z{\ensuremath{\mathbb{Z}}}
\newcommand{\Zp}{\ensuremath{\mathbb{Z}^{+}}}

\def\pd{\partial}

\newcommand{\mr}[1]{\mathrm{#1}}

\newcommand{\abs}[1]{\left|{#1}\right|}
\newcommand{\nrm}[1]{\left\|{#1}\right\|}
\newcommand{\set}[1]{\left\{#1\right\}\relax}
\newcommand{\scal}[1]{\left\langle\relax #1 \relax\right\rangle}
\newcommand{\qtxtq}[1]{\quad \text{#1}\quad}

\newtheorem{thm}{Theorem}[section]
\newtheorem{theorem}{Theorem}[section]
\newtheorem{proposition}[thm]{Proposition}

\newtheorem{corollary}[thm]{Corollary}

\title{An Integral Expression for the Dunkl Kernel in the 
Dihedral Setting}

\author{M. Maslouhi}
\address{M. Maslouhi: Ecole Nationale des Sciences Appliqu\'ees. Universit\'e Ibn Tofail, 14000 Kenitra, Maroc.}
\email {mostafa.maslouhi@univ-ibntofail.ac.ma}


\begin{document}

\maketitle

\markboth{M. Maslouhi}{An Integral Expression for the Dunkl Kernel in the 
Dihedral Setting}

\begin{abstract} 
	In this paper, we establish an integral expression for the Dunkl kernel in the context of Dihedral group of an arbitrary order by using the results in \cite{M-Y-Vk} where a construction of the Dunkl intertwining operator for a large set of regular parameter functions is provided. We introduce a differential system that leads to the explicit expression of the Dunkl kernel whenever an appropriate solution of it is obtained. In particular, an explicit expression of the Dunkl kernel $E_k(x,y)$ is given when one of its argument $x$ or $y$ is invariant under the action of a known reflection in the Dihedral group.  We obtain also a generating series for the homogeneous components $E_m(x,y)$, $m\in\Zp$, of the Dunkl kernel from which we derive new sharp estimates for the Dunkl kernel when the parameter function $k$ satisfies $\mr{Re}(k)>-\nu$,  $\nu$ an arbitrary nonnegative integer, which, up to our knowledge, is the largest context for such estimates so far.  
\end{abstract}

\section{Introduction And Preliminaries}\label{sec: prelim}
In \cite{DUNKL4}, Dunkl defined a family of first-order differential-difference operators associated to a Coxeter group. These operators generalize in a certain manner the usual differentiation and have gained  during recent years considerable interest in various fields of mathematics and also in physical applications (see \cite{DEJEU1,DEJEU3,LV} and the references therein). Our paper gives some new results around the Dunkl kernel \cite{DEJEU1} which is a key tool in this theory and whose expression is  unfortunately known explicitly only in few cases. 

In this paper, we use the results in \cite{M-Y-Vk} to establish an integral expression for the Dunkl kernel in the context of Dihedral group of an arbitrary order. This allows us to obtain the explicit expression of the Dunkl kernel $E_k(x,y)$ when one of its argument $x$ or $y$ is invariant under the action of a known reflection in the Dihedral group.  We obtain also a generating series for the homogeneous components $E_m(x,y)$, $m\in\Zp$, of the Dunkl kernel from which we derive new sharp estimates for the Dunkl kernel when the parameter function $k$ satisfies $\mr{Re}(k)>-\nu$,  $\nu$ an arbitrary nonnegative integer, which, up to our knowledge, is the largest context for such estimates so far. 

The paper is organized as follows. The remaining of this section serves to introduce concepts and notations needed for the sequel.  In section \ref{sec: case of D}, we set up the Dihedral group setting and we establish an important relationship between the homogeneous components $E_m$, $m\in\Zp$, of the Dunkl kernel, which turn to be crucial for our integral  representation. In section \ref{sec:a_fundamental_system},  we introduce and study a fundamental differential system whose solutions are closely related to the explicit expression of the Dunkl kernel. Section \ref{sec: integral expression} establishes the integral expression for the Dunkl kernel and finally section \ref{sec: applications} is devoted to some applications.

Fix a reduced root system  $R$ (see \cite{humphreys}) and consider its associated Coxeter group $G$ generated by the reflections $\s_\a$ where $\a\in R$ and $$\s_\a(x):=x-2\scal{x,\a}\a/\abs{\a}^2,\quad (x\in\R^d).$$
Here $\scal{\,,\,}$ denotes the canonical inner product in the space $\R^d$ and $\nrm{\,}$
its euclidean norm. We extend the form $\scal{\,,\,}$ to a bilinear form on $\C^d\times\C^d$ again denoted by $\scal{\,,\,}$.

Throughout the paper,  if $B\in \C^{d\times d}$, $\nrm{B}$ stands for the supremum norm of $B$.  

The action of $G$ on functions $f:\R^d\to \C$ is defined by  
\begin{equation}\label{eq: action of G}
L_{g}(f):= f\circ g^{-1},\quad g\in G.
\end{equation}

Set $R_+:=\{\a\in R:\,\scal{\a,\b}>0\}$ for some $\b\in\R^d$ such that $\scal{\a,\b}\ne 0$ for all $\a\in R$.

Let $k$ be a parameter function on $R$, that is,  $k:R\to \C$ and $G$-invariant.

 For $\xi\in\R^n$, the Dunkl operator $T_\xi$ on $\R^d$ associated to the group $G$ and the parameter function $k$ is acting on functions $f\in C^1(\R^d)$  by
$$  T_\xi(k) f(x)=\pd_\xi f(x)+\sum_{\a\in R_+}k(\a)\scal{\a,\xi}\frac{f(x)-f(\s_\a x)}{\scal{\a,x}},\quad (x\in\R^d).$$

Let $M$ be the vector space of all parameter functions and set $$M^{reg}:=\set{k\in M:\,
\cap_{\xi \in \R^d}\ker(T_\xi(k))=\C\cdot 1}.$$ 
$M^{reg}$ is called the set of regular parameter functions. 

We let $\mathcal{P}_m$, $m\in\Zp$, denotes the space of all homogeneous polynomials of degree $m$ in the $\C$-algebra $\Pi^d:=\C[\R^d]$.

 It has been shown in \cite{DUNKL5} and \cite{DJO} that for each $k\in M^{reg}$
there exists a unique isomorphism $V_k$ of $\Pi^d$, called the intertwining
operator, satisfying
\begin{equation}\label{eq: def principal of V_k}
    V_k(\mathcal{P}_n)\subset \mathcal{P}_n,\quad  V_k(1)=1\qtxtq{and} T_\xi V_k=V_k\pd_\xi, \quad
    \forall\xi\in\R^d.
\end{equation}

The Dunkl kernel or the generalized exponential $E_k(\cdot,\cdot)$, is a key tool in the Dunkl's theory and it was introduced \cite{DUNKL2} as the unique function satisfying $$E_k(0,y)=1,\quad T_\xi E_k(x,y)=\scal{\xi,y}E_k(x,y),\quad \xi,x\in\R^d,\quad  y\in\C^d.$$ Further details about this kernel may be found in \cite{DEJEU1,DUNKL2}.

The expression of the kernel $E_k$ is known explicitly only in few cases. Following \cite{M-Y-Vk} we have
\begin{equation*}
E_k(x,y)=\sum_{m=0}^\infty E_m(x,y),\quad x\in\R^d,\ y\in\C^d,
\end{equation*}  where 
\begin{equation}\label{eq: def of E_m}
E_m(x,y):=\frac{1}{m!}V_k(\scal{\cdot,y}^m)(x),
\end{equation}
for a large set $M^{*}$ of regular parameter functions $k$ presented later in this section. 

We recall from \cite{M-Y-Vk} some notations used 
along this paper. We consider the operator $A$ acting on functions  $f: \R^d \to \C$ by
\begin{equation}\label{eq: def of A}
A:=\sum_{\b\in R_+}k(\b) L_{\s_\b},
\end{equation} where $L_{g}$ is defined by \eqref{eq: action of G}.
 It is clear that $A$ leaves invariant the space $\mathcal{P}_m$ for all nonnegative integer $m$, thus $A_m:=A_{|\mathcal{P}_m}$ is an endomorphism of $\mathcal{P}_m$. 

Let $M^*$ denotes the set of all parameter functions for which the operator
$$H_m:=\left((m+\gm)-A_m\right)^{-1},\quad \gm:=\sum_{\a\in R_+}k(\a)$$ is defined for all $m\ge 1$. When $k\in M^*$, define the operator $H$ by $H_{|\mathcal{P}_m} := H_m$.

Appealing to \cite{M-Y-Vk}, $M^*$ includes the set of all parameter functions whose real part is nonnegative and $M^*\subset M^{reg}$. Moreover we have the crucial result: 
\begin{thm}\cite{M-Y-Vk}\label{thm: Main thm of V}\quad  Assume that $k\in 
M^{*}$. Let $m$ be a positive integer and $p\in \mathcal{P}_m$. Then  we have 
$$V_k(p)(x)=\sum_{j=0}^{d}x_jV_k(\pd_j (H \,p))(x), \quad (x\in\R^d).$$
\end{thm}

\section{The Dihedral Group Setting}\label{sec: case of D} 
In this section, we suppose that the Coxeter group $G$ associated to the root system $R$ is the Dihedral group $\mathcal{D}_n$ of order $2n$ where $n\ge 2$ is arbitrary. 

$\mathcal{D}_n$ is the symmetry group of a regular $n$-gons in the plane $\R^2$. Using the identification $\R^2=\C$ and letting $w:=e^{i\pi/n}$, then 
an associated positive root system is $R_+=\set{iw^j,j=0,\dots,n-1}$ while the rotations and reflections in $\mathcal{D}_n$ are given respectively by
 $r_j: z\mapsto z w^{2j}$ and $\s_j: z\mapsto \overline{z} w^{2j}$, $j=0,\dots, n-1$.

Taking $r=r_1$ and $\s =\s_0$, we may then write $\mathcal{D}_n=\mathcal{D}_n^+\cup \mathcal{D}_n^{-}$ with
\begin{equation}\label{eq: def of r and s}
	\mathcal{D}_n^{+}:=\set{r^j, j=0,\dots,n-1}\qtxtq{and}\mathcal{D}_n^{-}:=\set{r^j \s, j=0,\dots,n-1}.	
\end{equation}
It is worth to note that one has $\s r\s=r^{-1}$ and
\begin{equation}\label{eq: sum rotations null in D_n}
	\sum_{j=0}^{n-1} r^j x=0,
\end{equation}
for all $x\in\R^2$.  It is also important to keep in mind (see \cite[Theorem 4.2.7]{DUNKL7}) that in our setting the reflections in $\mathcal{D}_n$ are given exactly by the set $\mathcal{D}_n^{-}$.

In all the sequel, we fixe an arbitrary integer $n\ge 2$ and consider $\mathcal{D}_n$ as above. We will also suppose that the parameter function $k$ is constant, that is $k:=k(\a)$ for all $\a\in R_+$.
Although that in our setting we have $\gm=nk$, we will always write $\gm$ instead of $n k$ to conserve the standard notations for this theory.
  
With these settings in mind, the operator $A$ given by \eqref{eq: def of A} rewrites
\begin{equation}\label{eq: A in Dn}
	A=k\sum_{i=0}^{n-1}L_{r^j \s}. 
\end{equation}

The next result pursues the developments in \cite{M-Y-Vk} and gives more elements in the set $M^{*}$. Further, it makes concrete the action of $H_m$ on polynomials.
\begin{proposition}\label{pro: expr of Hm}
Consider  $k\in\C$ such that 
\begin{equation}\label{eq: k condition}
	  \gm\not\in\left\{-\frac{m}{2},\quad m\in\Z^{+}\right\}.
\end{equation} 
Then $k\in M^{*}$. Moreover, in this case the operator $H_m$ is given by
	\begin{equation*}
	H_m(p)(x)=\sum_{j=0}^{n-1} a_j(m)p(r^j x)	+\sum_{j=0}^{n-1}b_j(m) p(r ^j \s x),	
	\end{equation*}
  for all positive integer $m$, where for $j=0,\dots,n-1$, we set
	\begin{equation*}	a_j(m)=\delta_{j,0}\frac{1}{m+\gm}+\frac{\gm^2}{nm(m+\gm)(m+2\gm)} \qtxtq{and}	b_j(m)=\frac{\gm}{nm(m+2\gm)},		
	\end{equation*}
	with $\delta_{j,0}=1$ if $j=0$ and $\delta_{j,0}=0$ otherwise.
\end{proposition}
\begin{proof}
Take $k\in\C$ satisfying \eqref{eq: k condition} and fix any positive integer $m$. From \eqref{eq: A in Dn}, we see that $A=\gm L_{\s}R$ where $R:=\frac{1}{n}\sum_{j=0}^{n-1}L_{r^j}$ is a projector. 

Hence, each polynomial $p\in\mathcal{P}_m$ decomposes uniquely as $p=p_1+p_2$ where $p_1,p_2\in\mathcal{P}_m$ with
$A(p_1)=0$ and $A(p_2)=\gm L_\s p_2$. Therefore we get,
$$((m+\gm)-A)p_1=(m+\gm)p_1 \qtxtq{and} ((m+\gm)-A)p_2=((m+\gm)-\gm L_\s) p_2.$$
By view of  $\gm\not\in\left\{-\frac{m}{2},m\in\Z^{+}\right\}$, we may define the operator $\Lambda$ in $\mathcal{P}_m$ by
\begin{align*}
	\Lambda(p_1)&=\frac{1}{m+\gm}p_1,\\
	\Lambda (p_2)&=\left(\frac{m+\gm}{m(m+2\gm)}+\frac{\gm}{m(m+2\gm)}
	L_\s\right)p_2.	
\end{align*}
A straightforward calculation  shows that $W\circ ((m+\gm)-A)=\mr{Id}_{\mathcal{P}_m}$, thus $k\in M^{*}$ and $\Lambda=H_m$.  Further, since $\gm p_2=L_\s A(p)$, we get
\begin{align*}
H_m(p)&=\frac{1}{(m+\gm)}p+\frac{\gm^2}{m(m+\gm)(m+2\gm)}p_2+\frac{\gm}{m(m+2\gm)}L_\s p_2\\
&=\frac{1}{(m+\gm)}p+\frac{\gm}{m(m+\gm)(m+2\gm)}L_\s A (p)+\frac{1}{m(m+2\gm)}A(p),	
\end{align*}
and the proof is complete.
\end{proof}

From now on, we assume that the parameter function $k$ satisfies the condition \eqref{eq: k condition}.

Based on the definition \eqref{eq: def of E_m}, Theorem \ref{thm: Main thm of V} and a direct application of Proposition \ref{pro: expr of Hm} the proposition below gives a useful relationship between $E_{m+1}$ and $E_m$. 
\begin{proposition}\label{prop: relation E_m, E_{m+1}} For all  $m\in\Z^+$ and all $x,y\in\R^d$ we have
$$E_{m+1}(x,y)=\sum_{j=0}^{n-1}a_j(m+1) \scal{r^j x,y}E_{m}(r^j  x,y)+\sum_{j=0}^{n-1}b_j(m+1) \scal{r^j \s x,y}E_{m}(r^j \s x,y)
$$ where the constants $a_j,b_j$, $j=0,\dots,n-1$ are given by Proposition \ref{pro: expr of Hm}.
\end{proposition}

As an important consequence of Proposition \ref{prop: relation E_m, E_{m+1}}, with  \eqref{eq: sum rotations null in D_n} in mind, one has \begin{equation}\label{eq: expression of E_1}
	E_{1}(x,y)=\frac{1}{1+\gm}\scal{x,y}, \quad x,y\in\R^2.
\end{equation}

In the remaining of this section, we will proceed to construct a sequence that turn to be an essential ingredient for the expression of the Dunkl kernel. Before to go on, we need to introduce some additional notations.

For a positive integer $m$, we set $I_m:=\mathrm{diag}(1,\dots,1)\in \C^{m\times m}$. 

For $X\in\C^{m\times 1}$, we define its rank one associated matrix $X\otimes X\in \C^{m\times m}$ by $$X\otimes X:=X\times {}^{t}X.$$
For $x,y\in\R^2$, we define the diagonal matrix $D:=D(x,y)$ by 
\begin{equation}\label{eq: def of D}
	D(x,y):=\mathrm{diag}\left(\scal{r^j  x,y},\quad j=0,\dots,n-1\right)\in \C^{n\times n}.
\end{equation}
Finally, we define the one column vectors $W,W_s\in \C^{2n\times 1}$ by 
\begin{equation}\label{eq: def of W,Ws}
	W:={[1,\dots,1]}^{T}\qtxtq{and} W_s:={[1,\dots,1,-1,\dots,-1]}^{T}.
\end{equation}
In the sequel we will always identify $\C^{m\times 1}$ with $\C^{m}$, $m\ge 1$.

For $x,y\in\R^2$ consider the vector $X_{m}:=X_{m}(x,y)\in\C^{n}$ whose components are given by:
$$X_{m,i}(x,y)=E_m(r^i x,y),\quad i=0,\dots, n-1.$$ 

For the sake of simplicity,   we will omit the dependance on $y$ or on $x$ and $y$ together, when there is no need to reveal them. For instance, we will write, $X_{m,i}(x)$ in place of $X_{m,i}(x,y)$ and $D(x)$ or $D$ in place of $D(x,y)$ and so on.  

By Proposition \ref{prop: relation E_m, E_{m+1}} and using notation  modulo $n$ for the indices we get for $i=0,\dots, n-~1$: 
\begin{align*}
	X_{m+1,i}(x)&=E_{m+1}(r^i x,y)\\
	&=\sum_{j=0}^{n-1}a_j(m+1) \scal{r^{i+j} x,y}E_{m}(r^{i+j} x,y)+\sum_{j=0}^{n-1}b_j(m+1) \scal{r^{j-i} \s x,y}E_{m}(r^{j-i} \s x,y)\\
&=\sum_{j=0}^{n-1}a_{j-i}(m+1) \scal{r^{j} x,y}E_{m}(r^{j} x,y)+\sum_{j=0}^{n-1}b_{j+i}(m+1) \scal{r^{j} \s x,y}E_{m}(r^{j} \s x,y)\\
&=\sum_{j=0}^{n-1}a_{j-i}(m+1) \scal{r^{j} x,y}X_{m,j}(x)+\sum_{j=0}^{n-1}b_{j+i}(m+1) \scal{r^{j} \s x,y}X_{m,j}(\s x).	
\end{align*}
The n latter equations for $i=0,\dots, n-1$,  can be gathered in a matrix form as follows: 
\begin{equation}\label{eq: relation X(m+1) and X(m)}
	X_{m+1}(x)=\mathcal{A}_m D(x)X_m(x)+\mathcal{B}_m D(\s x)X_m(\s x)
\end{equation} where the matrices $\mathcal{A}_m$ and $\mathcal{B}_m$ are defined by
\begin{equation}\label{eq: def of Am}
\mathcal{A}_m:=\frac{1}{m+1+\gm}I_n+\frac{\gm^2}{n(m+1)(m+1+\gm)(m+1+2\gm)}	\mathcal{O}_n,
\end{equation}
\begin{equation}\label{eq: def of Bm}
	\mathcal{B}_m:=\frac{\gm}{n(m+1)(m+1+2\gm)}\mathcal{O}_n.
\end{equation}
where $\mathcal{O}_n$ is the square matrix of order $n$ whose entries are all equal to 1. This leads to
 $$	\begin{bmatrix}
X_{m+1}(x,y)\\X_{m+1}(\s x,y)	
\end{bmatrix}=\mathcal{W}_m D(x,y) 
	\begin{bmatrix}
X_m(x,y)\\X_m(\s x,y)	
\end{bmatrix}$$
where $\mathcal{W}_m $ is defined by
\begin{align*}
\mathcal{W}_m&:=\frac{1}{m+1+\gm}I_{2n}+\frac{\gm}{n(m+1)(m+1+2\gm)}	\begin{bmatrix}
		\frac{\gm}{m+1+\gm}\mathcal{O}_n &\mathcal{O}_n\\\mathcal{O}_n&\frac{\gm}{m+1+\gm}\mathcal{O}_n 
	\end{bmatrix}.
\end{align*}
Noting that 
$$\begin{bmatrix}
		\frac{\gm}{m+1+\gm}\mathcal{O}_n &\mathcal{O}_n\\\mathcal{O}_n&\frac{\gm}{m+1+\gm}\mathcal{O}_n 
	\end{bmatrix}=\frac{m+1+2\gm}{2(m+1+\gm)}W\otimes W-\frac{m+1}{2(m+1+\gm)}W_s\otimes W_s,$$ and setting 
		\begin{equation}\label{eq: def of Ym}
			Y_m:=Y_m(x,y)=(1+\gm)_m\begin{bmatrix}
		X_m(x,y)\\X_m(\s x,y)	
		\end{bmatrix}, 
		\end{equation} 
		where $(1+\gm)_m$ is the Pochhammer symbol, we get the starting useful relationship:
\begin{equation}\label{eq: Y(m+1) by Y(m)}
Y_{m+1}=D Y_m + \frac{\gm}{2n(m+1)}\scal{DY_m,W} W-\frac{\gm}{2n(m+1+2\gm)} \scal{DY_m,W_s} W_s,
\end{equation}
for all $m\in\Z^+$. Next, an easy induction gives   
\begin{equation}\label{eq: Y(m+1) by Y(m) developped}
	\begin{split}
		Y_{m}=D^{m-1} Y_1 &+\frac{\gm}{2n}\sum_{j=1}^{m}\frac{1}{j} \scal{DY_{j-1},W}D^{m-j} W\\ \quad &-\frac{\gm}{2n}\sum_{j=1}^{m} \frac{1}{j+2\gm}\scal{DY_{j-1},W_s}D^{m-j}W_s,
\end{split}
\end{equation}for all positive integer $m$. 

From \eqref{eq: expression of E_1} and \eqref{eq: Y(m+1) by Y(m)} we get $Y_1=DW$ and the equation  \eqref{eq: Y(m+1) by Y(m) developped} rewrites
\begin{equation}\label{eq: Y(m+1) by Y(m) pre-final}
	\begin{split}
		Y_{m}=D^{m} W &+\frac{\gm}{2n}\sum_{j=1}^{m}\frac{1}{j} \scal{DY_{j-1},W}D^{m-j} W\\ \quad &-\frac{\gm}{2n}\sum_{j=1}^{m} \frac{1}{j+2\gm}\scal{DY_{j-1},W_s}D^{m-j}W_s,
\end{split}
\end{equation}
for all positive integer $m$. Taking the scalar products of the  both members in \eqref{eq: Y(m+1) by Y(m) pre-final} with $DW$ and $DW_s$ respectively, we get
\begin{equation}\label{eq: Y(m+1) by Y(m) final-1}
	\begin{split}
	\scal{Y_{m},D W}&=\scal{D^{m+1}W ,W} +\frac{\gm}{2n}\sum_{j=1}^{m}\frac{1}{j}
	 \scal{Y_{j-1},DW}\scal{D^{m-j+1} W,W}\\
	&\quad-\frac{\gm}{2n}\sum_{j=2}^{m}\frac{1}{j+2\gm} \scal{DY_{j-1},W_s}\scal{D^{m-j+1}W_s,W},
\end{split}
\end{equation}and
\begin{equation}\label{eq: Y(m+1) by Y(m) final-2}
	\begin{split}
	\scal{Y_{m},D W_s}&=\scal{D^{m+1}W_s ,W} +\frac{\gm}{2n}\sum_{j=1}^{m}\frac{1}{j}
 \scal{Y_{j-1},DW}\scal{D^{m-j+1} W_s,W}\\
&\quad-\frac{\gm}{2n}\sum_{j=2}^{m}\frac{1}{j+2\gm} \scal{DY_{j-1},W_s}\scal{D^{m-j+1}W,W},
\end{split}
\end{equation}
for all $m\ge 1$.  The equations \eqref{eq: Y(m+1) by Y(m) final-1} and \eqref{eq: Y(m+1) by Y(m) final-2} are the fundamental pieces of our main result.


\section{The Group Associated Differential System} 
\label{sec:a_fundamental_system}
In this section we present and study a differential system related to the Dihedral group $\mathcal{D}_n$, which we will refer to in this paper, simply, as the group differential system since. It will be shown in section \ref{sec: integral expression} and \ref{sec: applications}, that the explicit expression of the Dunkl kernel as well as  its integral representation are closely related to this differential system. More precisely, we get the explicit expression of the Dunkl kernel whenever we get an appropriate solution of the homogeneous differential system associated to the group differential system.

Before to pursue our developments, we need to introduce some specific additional notations here.  Recall that for $B\in \C^{2\times 2}$, $\nrm{B}$ will denotes its supremum norm, while for $X\in \C^{2}$, $\nrm{X}$ stands for its euclidean norm.  

In all the sequel we fix $x,y\in \R^2$ and we define
 \begin{equation}\label{eq: def of a(x,y)}
 	a(x,y):=\max\set{\abs{\scal{g x,y}},\quad g\in\mathcal{D}_n}. 
 \end{equation}

In order to include the case $a(x,y)=0$ into our developments (especially when $x=0$ or $y=0$), we will allow  $\frac{1}{a(x,y)}=+\infty$ when $a(x,y)=0$.

For $0<\rho\le +\infty$, $\D_{\rho}$ is the open disk in the complex plane $\C$ centered at 0 with radius $\rho$ if $\rho<+\infty$ and the hole plane $\C$ otherwise. We let also $\D_{\rho}^{*}:=\D_{\rho}-\set{0}$ denotes its associated punctured disk.

Define the functions $g$ and $g_s$ by
\begin{align}
	g(z):=g(z,x,y)&=\sum_{i=0}^{n-1} \frac{\scal{r^ix,y}}{1-z\scal{r^i x,y}}+\sum_{i=0}^{n-1} \frac{\scal{r^i\s x,y}}{1-z\scal{r^i \s x,y}},\label{ali: def of g}\\
	g_s(z):=g_s(z,x,y)&=\sum_{i=0}^{n-1} \frac{\scal{r^ix,y}}{1-z\scal{r^i x,y}}-\sum_{i=0}^{n-1} \frac{\scal{r^i\s x,y}}{1-z\scal{r^i \s x,y}}\label{ali: def of gs}.
\end{align}
Finally, consider the $\C^2$-valued function $\Lambda$ given by
\begin{equation}\label{eq: def of matrix Lambda}
	\Lambda(z):=\Lambda(z,x,y,z)=\frac{\gm}{2n}\begin{bmatrix}
		g(z)&-g_s(z)\\g_s(z)&-g(z)
	\end{bmatrix}.
\end{equation}
$\Lambda$ is a $\C^2$-holomorphic function in $\D_{\frac{1}{a(x,y)}}$ with 
 $h(z)=\sum_{p=0}^{\infty}B_p(x,y) z^p$ and
\begin{equation}\label{eq: def matrix Bp}
	B_p(x,y):=\frac{\gm}{2n}\sum_{i=0}^{n-1}\begin{bmatrix}
		 \scal{r^ix,y}^{p+1}+\scal{r^i \s x,y}^{p+1}&- \left(\scal{r^ix,y}^{p+1}-\scal{r^i \s x,y}^{p+1}\right)\\[3pt]  \scal{r^ix,y}^{p+1}-\scal{r^i \s x,y}^{p+1}&  -\left(\scal{r^ix,y}^{p+1}+\scal{r^i \s x,y}^{p+1}\right)
	\end{bmatrix}.
\end{equation}

From now on, for the sake of simplicity, we will always omit the dependance on $x$ or $y$ or both when there is no need to reveal them.

Consider the group differential system
\begin{equation}\label{eq: full system}
X^{\prime}(z)=\begin{bmatrix}
	g(z)\\[3pt] g_s(z)
\end{bmatrix} +	\begin{bmatrix}
		\frac{\gm}{2n} g(z)&-\frac{\gm}{2n} g_s(z)\\[3pt]
		\frac{\gm}{2n} g_s(z)&-\left(\frac{2\gm}{z}+\frac{\gm}{2n}  g(z)\right),
	\end{bmatrix}X(z)
\end{equation}
and its associated homogeneous system
\begin{equation}\label{eq: Homogen-System}
	X^{\prime}(z)=\begin{bmatrix}
		\frac{\gm}{2n} g(z)&-\frac{\gm}{2n} g_s(z)\\[3pt]
		\frac{\gm}{2n} g_s(z)&-\left(\frac{2\gm}{z}+\frac{\gm}{2n}  g(z)\right)
	\end{bmatrix}X(z).		
\end{equation}

An important tool for our developments is given by the next proposition.
\begin{proposition}\label{pro: uniquness of vanishing sols}
Fix $R>0$. Let $X$ be a $\C^2$-valued holomorphic function in $\D_{R}$ solution of \eqref{eq: Homogen-System} in $\D_{R}^{*}$ and vanishing at $z=0$. Then $X=0$ in $\D_{R}$. 
\end{proposition}
\begin{proof}
Let $R>0$ and $X$ be as in the Proposition \ref{pro: uniquness of vanishing sols}. We have then 
\begin{equation}\label{eq: gen eq}
zX^{\prime}(z)=-2\gm J \times X(z)+ z\Lambda(z)\times X(z),
\end{equation}
for all $z\in\D_{R}^{*}$. Here $\Lambda$ is the function defined by \eqref{eq: def of matrix Lambda} and $J:=\begin{bmatrix}0&0\\0&1\end{bmatrix}$.

Since $JX(0)=0$ we see that \eqref{eq: gen eq} holds for $z=0$ as well.

In other hand, by analyticity argument, there exists a sequence $A_p:=A_p(x,y)\in \C^2$, $p\in \Zp$ satisfying $X(z)=\sum_{p=0}^{+\infty}A_p z^p$  for all  $z\in \D_{R}$ and an easy calculation shows  that \eqref{eq: gen eq} equivalents to
\begin{equation}\label{eq: power series eq}
	\sum_{p=1}^{+\infty}p A_{p} z^p=2\gm J\sum_{p=0}^{+\infty}A_{p} z^p+\sum_{p=0}^{+\infty}C_p z^{p+1},
\end{equation}
where $C_p=\sum_{k=0}^{p}B_kA_{p-k}$ and the matrices $B_p$ are given by \eqref{eq: def matrix Bp}. 
 
View of \eqref{eq: sum rotations null in D_n} we see that $B_0=0$ and then \eqref{eq: power series eq} rewrites
$$\sum_{p=1}^{+\infty}p A_{p} z^p=2\gm J\sum_{p=0}^{+\infty}A_{p} z^p+\sum_{p=1}^{+\infty}C_{p-1} z^p,$$
for all $z\in\D_R$, which in turn, since $JA_0=0$, equivalents to
\begin{equation}\label{eq: induction rel1}
	A_{p}=(p+2\gm J)^{-1}\sum_{k=0}^{p-1}B_{p-i-1}A_{i},\quad (p\ge 1).
\end{equation}
Note that by our assumption on the parameter function $k$, the matrix $(p+2\gm J)^{-1}$ is well defined and an induction yields $A_p=0$ for all $p\in\Zp$.
\end{proof}

The next theorem is the last rock in our developments and is the fundamental one.
\begin{theorem}\label{thm: holomorphic sols}
There exists a constant $\delta\ge 1$ not depending on $(x,y)$ and a $\C^2$-valued holomorphic function $Q$ in $\D_{\frac{1}{\delta\,a(x,y)}}$  vanishing at $z=0$ and solution of \eqref{eq: full system} in $\D_{\frac{1}{\delta a(x,y)}}^{*}$.
$Q$ is defined by $Q(z)=\sum_{p=1}^{+\infty}A_p z^p$ where $A_p:=A_p(x,y)\in\C^2$ given by:
\begin{equation}\label{eq: def of Ap-2}
	A_0:=\left(\frac{2n}{\gm},0\right),\quad A_{p}:=\begin{bmatrix}
		\frac{1}{p}&0\\0&\frac{1}{p+2\gm}
	\end{bmatrix}\sum_{i=0}^{p-1}B_{p-i-1}A_{i},\quad (p\ge 1).
\end{equation}Here the matrices $B_p$ are given by \eqref{eq: def matrix Bp}.

The coefficient $A_p$ satisfies
\begin{equation}\label{eq: main inequality for Ap}
	\nrm{A_{p}}\le \frac{2n}{\abs{\gm}} \left(\delta  a(x,y)\right)^{p}. 
\end{equation} for all $p\in\Zp$.
Moreover, each $\C^2$-valued holomorphic function in $\D_\rho$, $\rho>0$, solution of \eqref{eq: full system} in $\D_\rho^{*}$ and vanishing at $z=0$ coincides with $Q$ in the disk $\D_{\rho}\cap \D_{\frac{1}{\delta a(x,y)}}$.
\end{theorem}
\begin{proof}
Assume for now that $Q$ exists and is as in  Theorem \ref{thm: holomorphic sols}. Then since $Q(0)=0$, the last part of Theorem \ref{thm: holomorphic sols} follows immediately from Proposition \ref{pro: uniquness of vanishing sols}.

Now we will proceed to construct the function $Q$.   
First of all note that the relationship  \eqref{eq: def of Ap-2}
define a unique sequence $A_p\in\C^2$ depending only in $(x,y)$ and, view of \eqref{eq: def matrix Bp}, an easy induction shows that the components of $A_p(x,y)$, $p\in\Zp$ are in fact homogeneous polynomials of degree $p$ with respect to $x$ and $y$.
   
Further,  it is easy to see from \eqref{eq: def matrix Bp} that 
\begin{equation}\label{eq: estim for nrm(Bp)}
	\nrm{B_p}\le \abs{\gm}\left( a(x,y)\right)^{p+1},
\end{equation}
for all $p\in \Z^{+}$ and then 
\begin{equation}\label{eq: inequality Ap}
	\nrm{A_{p}}\le \frac{\delta }{p} \sum_{i=0}^{p-1}\left( a(x,y)\right)^{p-i}\nrm{A_{i}},
\end{equation}
where we have set $$\delta:=2\abs{\gm}\sup_{p\in \Zp}\set{1, \frac{p}{\abs{p+2\gm}}}.$$
Note that up to consider $\max(1,\delta)$, we may and will supose that $\delta\ge 1$. 

Suppose for a while that $ a(x,y)>0$. Then since $\delta\ge 1$,  the inequality \eqref{eq: inequality Ap} yields
\begin{equation*}
	\frac{\nrm{A_{p}}}{\left(\delta a(x,y)\right)^{p}}\le \frac{1 }{p}\sum_{i=0}^{p-1}\frac{\nrm{A_{i}}}{\left(\delta a(x,y)\right)^i}, 
\end{equation*}
and an immediate induction gives
\begin{equation}\label{eq: final ineq Ap}
	\nrm{A_{p}}\le \left(\delta a(x,y)\right)^{p} \nrm{A_{0}},
\end{equation}
for all $p\in \Zp$. View of \eqref{eq: inequality Ap}, we see that \eqref{eq: final ineq Ap} holds also when $ a(x,y)=0$.

Further, the function $Q$ defined by $Q(z)=\sum_{p=1}^{+\infty}A_p z^p$ is holomorphic in $\D_{\frac{1}{\delta a(x,y)}}$, vanishes at $z=0$ and the same arguments used in Proposition \ref{pro: uniquness of vanishing sols} show that $Q$ is solution of \eqref{eq: full system} in $\D_{\frac{1}{\delta a(x,y)}}^{*}$, which completes the proof.
\end{proof}
The next result solves the group differential system in the case where $x$ or $y$ is $\s$-invariant which allows us, in section \ref{sec: integral expression}, to obtain the explicit expression of the Dunkl kernel in this context.
\begin{corollary}\label{cor: invariant-setting}
	Assume that $x$ or $y$ is $\s$-invariant. Then the function $Q$ defined in Theorem \ref{thm: holomorphic sols} is given by
	$$Q(z)=-\frac{2}{k}\left(1-\prod_{i=0}^{n-1}(1-z\scal{r^i x,y})^{-k},0\right),\quad z\in\D_{\frac{1}{\delta a(x,y)}}$$ for some $\delta\ge 1$, where we make use of the principal determination of the $z\mapsto (1-z)^{-k}$ in the unit disk $\D$.
\end{corollary}
\begin{proof}	
Suppose that $x$ is $\s$-invariant, then we get $g_s(z)=0$ and $$g(z)=2\sum_{i=0}^{n-1} \frac{\scal{r^ix,y}}{1-z\scal{r^i x,y}},$$
	where $g$ and $g_s$ are given by \eqref{ali: def of g} and \eqref{ali: def of gs}.  In this case, is it easy to verify that the $\C^2$-holomorphic function 
	$$X(z):=-\frac{2}{k}\left(1-\prod_{i=0}^{n-1}(1-z\scal{r^i x,y})^{-k},0\right),$$
	is  solution of \eqref{eq: full system} in $\D_{\frac{1}{a(x,y)}}$ with $X(0)=0$. 
	
Since $g_s(z,x,y)=g_s(z,y,x)$ and $g(z,x,y)=g(z,y,x)$ then the same result holds when $\s y=y$. Now Theorem \ref{thm: holomorphic sols} ends the proof.
\end{proof}


\section{An Integral Expression for the Dunkl kernel} 
\label{sec: integral expression}
This section is the last step towards the main result of this paper.  We will gather here all the work done in the previous sections to derive out an integral expression for the polynomials $E_m$, $m\in\Zp$, which, as a consequence, provides  an integral expression for the Dunkl kernel. 
We mention here that in \cite{AMRI-BECHIR}, an integral representation of $E_k$ was obtained for the rank-two root system $A_2$. Also, based on the results of \cite{M-Y-Vk} and \cite[Chapter 3]{M-Thesis}, the authors in \cite{N-L-Y} present another version of \cite{AMRI-BECHIR}. 

Let  $Q(z,x,y):=(Q_1(z,x,y),Q_2(z,x,y))\in\C^2$ be as in Theorem \ref{thm: holomorphic sols} and set
 \begin{equation}\label{eq: def of Phi}
 \Phi(z,x,y):=\frac{2n}{\gm}+Q_1(z,x,y)-Q_2(z,x,y).
 \end{equation}

Our first main result is the generating series of the polynomials $E_m$ given below.
 \begin{theorem}\label{thm: generating series of Em}There exists a constant $\delta\ge1$  such that for all $x,y\in\R^2$ and  $0<\rho<\frac{1}{\delta a(x,y)}$ we have	
	 \begin{equation}\label{eq: integ expr of Em}	 	E_m(x,y)=\frac{\gm/2n}{(1+\gm)_m}\frac{1}{2i\pi}\int_{\abs{z}=\rho}\frac{\Phi(z,x,y)}{1-z\scal{x,y}}\frac{dz}{z^{m+1}}.
	 \end{equation}
		Here $\Phi$ is defined by \eqref{eq: def of Phi}. By consequence we get
   	 \begin{equation}\label{eq: generating series for Em}	 	\frac{\Phi(z,x,y)}{1-z\scal{x,y}}=\frac{2n}{\gm}\sum_{m=0}^{+\infty}(1+\gm)_mE_m(x,y).
   	 \end{equation}
 \end{theorem}
\begin{proof}
We use notations \eqref{eq: def of D} and \eqref{eq: def of W,Ws} where, once again, we omit the dependance on $(x,y)$. Consider the functions $f$, $f_s$, $F$ and $F_s$ defined by their respective power series where as follows:
\begin{alignat*}{2}
	f(z)&=\sum_{j=1}^\infty \scal{Y_{j},D W}z^{j}&&,\quad
	f_s(z)=\sum_{j=1}^\infty \scal{Y_{j},D W_s}z^{j}.\\
	F(z)&=\sum_{j=1}^\infty \frac{1}{j} \scal{Y_{j-1},D W}z^j&&,\quad 
	F_s(z)=\sum_{j=1}^\infty \frac{1}{j+2\gm} \scal{Y_{j-1},D W_s}z^j.
\end{alignat*}
If we let $\delta=\delta(n,k):=\max_{m\in\Z_p}(\nrm{\mathcal{A}_m}+\nrm{\mathcal{B}_m})$, where the matrices $\mathcal{A}_m$ and $\mathcal{B}_m$ are given by   \eqref{eq: def of Am} and \eqref{eq: def of Bm}, then we see from  \eqref{eq: def of Ym} and \eqref{eq: relation X(m+1) and X(m)} that $$\nrm{Y_{m+1}}\le \delta a(x,y)\nrm{Y_{m}},$$ 
for all $m\in\Zp$. We deduce that $$\nrm{Y_{m}}\le \left(\delta a(x,y)\right)^m\nrm{Y_{0}}=\left(\delta a(x,y)\right)^m,$$ which leads to $$\abs{\scal{Y_{j-1},D W}}\le \left(\delta a(x,y)\right)^j \qtxtq{and} \abs{\scal{Y_{j-1},D W_s}}\le \left(\delta a(x,y)\right)^j,$$ for all $j\ge 1$.
This shows in particular that $F$ and $F_s$ are holomorphic functions in $\D_{\frac{1}{\delta a(x,y)}}$ satisfying $F(0)=F_s(0)=0$. Further, a straightforward calculations yields
$$F^{\prime}(z)=f(z),\quad F_s^{\prime}(z)=f_s(z)-\frac{2\gm}{z}F_s(z)$$ 
for all $z\in\D_{\frac{1}{\delta a(x,y)}}$.  From the equations \eqref{eq: Y(m+1) by Y(m) final-1} and \eqref{eq: Y(m+1) by Y(m) final-2} we see  that $(F,F_s)$ is solution  of  the differential system \eqref{eq: full system} in $\D_{\frac{1}{\delta a(x,y)}}^{*}$.

Note that , up to increase $\delta$, one gets by use of  Theorem \ref{thm: holomorphic sols} that $(F(z),F_s(z))=Q(z)$ for all $z\in\D_{\frac{1}{\delta a(x,y)}}$. 

 By their definitions, we have for all $j\ge 1$ and for all   $0<\rho< \frac{1}{\delta a(x,y)}$ that
\begin{align*}
	\frac{1}{j}\scal{Y_{j-1},DW}&=\frac{1}{2i\pi}\int_{\abs{z}=\rho}\frac{F(z)}{z^{j+1}}dz=\frac{1}{2i\pi}\int_{\abs{z}=\rho}\frac{G_1(z)}{z^{j+1}}dz,
\end{align*}
and 
\begin{align*}
	\frac{1}{j+2\gm} \scal{Y_{j-1},D W_s}&=\frac{1}{2i\pi}\int_{\abs{z}=\rho}\frac{F_s(z)}{z^{j+1}}dz=\frac{1}{2i\pi}\int_{\abs{z}=\rho}\frac{G_2(z)}{z^{j+1}}dz.
\end{align*}
All of this yields to
\begin{align*}
			Y_{m}&=D^{m} W +\frac{\gm}{2n}\sum_{j=1}^{m} \left(\frac{1}{2i\pi}\int_{\abs{z}=\rho}\frac{G_1(z)}{z^{j+1}}dz\right) D^{m-j} W\\
&-\frac{\gm}{2n}\sum_{j=1}^{m} \left(\frac{1}{2i\pi}\int_{\abs{z}=\rho} \frac{G_2(z)}{z^{j+1}} dz\right) D^{m-j}W_s.	
\end{align*} 
Finally, taking the scalar product with $(1,0,\dots,0)\in\R^{2n}$ for both members above, we get
\begin{align*}
			(1+\gm)_m E_m(x,y) &=\scal{x,y}^{m}+\frac{\gm}{2n}\sum_{j=1}^{m} \left(\frac{1}{2i\pi}\int_{\abs{z}=\rho} \frac{G_1(z)-G_2(z)}{z^{j+1}} dz\right)\scal{x,y}^{m-j}\\
&=\scal{x,y}^{m}+\frac{\gm}{2n}\sum_{j=1}^{m} \left(\frac{1}{2i\pi}\int_{\abs{z}=\rho} \frac{\Phi(z)}{z^{j+1}} dz\right)\scal{x,y}^{m-j}.
\end{align*}Whence
\begin{equation}\label{eq: expr od Em by integral}
(1+\gm)_m E_m(x,y)=\frac{\gm}{2n}\sum_{j=0}^{m} \left(\frac{1}{2i\pi}\int_{\abs{z}=\rho} \frac{\Phi(z)}{z^{j+1}} dz\right)\scal{x,y}^{m-j}.
\end{equation}
The last equation is obtained noting that $\frac{1}{2i\pi}\int_{\abs{z}=\rho} \frac{\Phi(z)}{z} dz=\Phi(0)=\frac{2n}{\gm}$. The result follows noting that the right member in \eqref{eq: expr od Em by integral} is nothing but the coefficient of order $m$ in the Cauchy product of $\frac{\gm}{2n}\Phi(z)$ with the power series $\frac{1}{1-z\scal{x,y}}
		=\sum_{j=0}^{\infty} \scal{x,y}^{j} z^j$. 
\end{proof}

In the sequel we will define $\mathcal{J}:=\mathcal{J}(k,n)$ as the set of all real numbers $\delta\ge 1$ satisfying Theorem \ref{thm: generating series of Em}. View of Theorem \ref{thm: generating series of Em} and the analyticity of the function $z\mapsto \Phi(z)$, it is not hard to see that actually  $\mathcal{J}$ is an interval of the form $(\inf \mathcal{J},+\infty[$.

\medskip

For $\delta\in \mathcal{J}$, $x,y\in\R^2$ and  $0<\rho<\frac{1}{\delta a(x,y)}$ define the kernel
\begin{equation}\label{eq: def of kernel K}
	\mathrm{K}(t,\delta,\rho,x,y):=\frac{\gm^2}{2n}\frac{1}{2i\pi}\int_{\abs{z}=\rho}\frac{\Phi(z,x,y)}{z(1-z\scal{x,y})}e^{t/z} dz,
	\end{equation}where $\Phi$ is defined by \eqref{eq: def of Phi}. Our integral representation for the Dunkl kernel now reads:

\begin{theorem}\label{thm: integral repres for Ek}
 Assume that $\mathrm{Re}(\gm)>0$. Then for all $\delta\in \mathcal{J}$, $x,y\in\R^2$ and  $0<\rho<\frac{1}{\delta a(x,y)}$ we have
	$$	E_k(x,y)=\int_0^1(1-t)^{\gm-1}\mathrm{K}(t,\delta,\rho,x,y)dt,$$ where the kernel $\mathrm{K}$ is given by \eqref{eq: def of kernel K}
\end{theorem}
\begin{proof} Assume that $k\in M^*$ and fix  $\delta\in \mathcal{J}$, $x,y\in\R^2$. Let $\rho$ be any positive real number such that $\rho<\frac{1}{\delta a(x,y)}$. 
Appealing to \cite{M-Y-Vk} and Theorem \ref{thm: generating series of Em} we have 
\begin{align*}
	E_k(x,y)&=\sum_{m=0}^{\infty}E_m(x,y)\\
	&=\frac{\gm}{2n}\frac{1}{2i\pi}\int_{\abs{z}=\rho}\frac{\Phi(z,x,y)}{z(1-z\scal{x,y})}\left(\sum_{m=0}^{\infty}\frac{1}{(1+\gm)_m  }\frac{1}{ z^{m}}\right)dz
\end{align*}
Now if we assume further that $\mathrm{Re}(\gm)>0$, then from \cite[p. 505]{abramowitz-handbook}, we get
\begin{align*}
	E_k(x,y)&=\frac{\gm^2}{2n}\frac{1}{2i\pi}\int_{\abs{z}=\rho}\frac{\Phi(z,x,y)}{z(1-z\scal{x,y})}\left(\int_{0}^1(1-t)^{\gm-1}e^{\frac{t}{z}}dt\right)dz.
\end{align*} which completes the proof.		
\end{proof}


\section{Applications} 
\label{sec: applications}
This section gives two important applications of the integral representation established in the previous section for the polynomials $E_m$, $m\in\Zp$. Namely, we will derive out the explicit expression of $E_m(x,y)$ when one of its arguments $x$ or $y$ is $\s$-invariant and we will also give sharp estimates for the Dunkl kernel when $\mr{Re}(\gm)>-\nu$ and $\nu\in \Z_p$.

We begin by the explicit expression, and for simplicity, we write here $\gm/n=k$ when it is needed.   
\begin{theorem}\label{thm: expr in teh s-invariant case}
Let $x,y\in\R^2$. Assume that $x$ or $y$ is $\s$-invariant then for all nonnegative integer $m$ we have
\begin{align*}
 E_m(x,y)&=\frac{1}{(1+\gm)_m}\sum_{j=0}^{m} 
 \left(\sum_{\nu_0+\dots+\nu_{n-1}=j}\prod_{i=0}^{n-1}\frac{(k)_{\nu_i}}{\nu_i!}\scal{r^i
  x,y}^{\nu_i}\right) \scal{x,y}^{m-j}.
 \end{align*}
\end{theorem}
\begin{proof}
Assume that $x$ is $\s$-invariant. From Corollary \ref{cor: invariant-setting} and \eqref{eq: def of Phi}, we see that
	 \begin{align*}
	 \Phi(z,x,y)&=\frac{2}{k}\prod_{i=0}^{n-1}(1-z\scal{r^{i}x,y})^{-k}\\
	&=\frac{2}{k}\prod_{i=0}^{n-1}\sum_{\nu_i=0}^{\infty}\frac{(k)_{\nu_i}}{\nu_i!}\scal{r^ix,y}^{\nu_i}z^{\nu_i}.
	 \end{align*}
	Whence
	$$\Phi(z,x,y)=\frac{2}{k}\sum_{m=0}^\infty\left(\sum_{\nu_0+\dots+\nu_{n-1}=m}\prod_{i=0}^{n-1}\frac{(k)_{\nu_i}}{\nu_i!}\scal{r^i
	  x,y}^{\nu_i}\right)z^m$$
	and Theorem \ref{thm: generating series of Em} ends the proof.
\end{proof}

The next theorem gives new sharp estimates for the homogeneous components of the Dunkl kernel $E_k$ when $\mr{Re}(\gm)>-\nu$, $\nu\in\Zp$, which are used next to derive  appropriate ones for $E_k$. We point out here that in \cite{ROS4} and \cite{DEJEU1} other concise estimates for $E_k$ are given when $k$ is  nonnegative or with nonnegative real part respectively.
\begin{theorem}\label{thm: estimates of Em}
	Assume that $\mr{Re}(\gm)>-\nu$ where $\nu$ is a nonnegative integer. Then there exists a constant $\delta\ge 1$ and  a positive constant $C:=C(n,k,\nu)$ such that 
$$\abs{E_m(x,y)}\le C \frac{m^{\nu+2}}{m!}(\delta a(x,y))^m, \quad (m\ge 1),$$ for all $x\in\R^2$, $y\in\C^2$.
\end{theorem}
\begin{proof}
	First of all by analyticity argument, Theorem \ref{thm: generating series of Em} holds for $x\in\R^2$ and $y\in\C^2$ as well. In other hand, recalling  the definition of the function $Q$  given   by Theorem \ref{thm: holomorphic sols}, we know that there exists a constant $\delta\ge 1$ such that
	\begin{align*}
		\nrm{Q(z)}&\le 	\sum_{p=1}^{\infty}\nrm{A_{p}}\abs{z}^p\le \sum_{p=1}^{\infty}\left(\delta a(x,y)\right)^{p}\abs{z}^p \nrm{A_{0}}= \frac{\delta a(x,y)\abs{z}}{1-\delta a(x,y)\abs{z}}\nrm{A_{0}},
	\end{align*}whenever $\delta a(x,y)\abs{z}<1$. Fix this constant $\delta$. Since $A_0=(\frac{2n}{\gm},0)$, this entails 
	\begin{align*}
		\abs{\Phi(z)}&\le \abs{\frac{2n}{\gm}+Q_1(z)}+\abs{Q_2(z)}\le \frac{1}{1-\delta a(x,y)\abs{z}}\nrm{A_{0}}+\frac{\delta a(x,y)\abs{z}}{1-\delta a(x,y)\abs{z}}\nrm{A_{0}}.
	\end{align*}
	Whence \begin{equation}\label{eq: estimates of M}
		\abs{\Phi(z)}\le \frac{4n/\abs{\gm}}{1-\delta a(x,y)\abs{z}},
	\end{equation} whenever $\delta a(x,y)\abs{z}<1$. From Theorem \ref{thm: generating series of Em} we infer that 
\begin{align*}		
	\abs{E_m(x,y)}&\le \frac{2}{\abs{(1+\gm)_m}}\frac{1}{(1-\rho \delta a(x,y))^2}	\frac{1}{\rho^m},	
\end{align*}
for all $\rho>0$ such that $\rho\delta a(x,y)<1$.  Put $\rho=\frac{t}{\delta a(x,y)}$,  to get that
\begin{align*}		
	\abs{E_m(x,y)}&\le \frac{2}{\abs{(1+\gm)_m}}\frac{(\delta a(x,y))^m}{t^m(1-t)^2},	
\end{align*}for all $t\in ]0,1[$. Taking the infimum over all  $t\in ]0,1[$, we obtain
	\begin{equation}\label{eq: inequality for Em, gm}
		\abs{E_m(x,y)}\le \frac{e^2}{2} \frac{(m+2)^2}{\abs{(1+\gm)_m}}(\delta a(x,y))^m,
	\end{equation}
for all $m\in\Zp$. Now, assume that $\mr{Re}(\gm)>-\nu$ for some $\nu\in\Zp$ fix a positive integer $m$. If $\nu=0$, then it is clear  that $\abs{(1+\gm)_m}\ge m!$. Suppose that $\nu\ge 1$, then we get $$\abs{(1+\gm)_m}\ge (m-\nu)!\prod_{j=1}^{\nu}\abs{\gm+j},$$ for all $m\ge \nu+1$, and then for both cases, from \eqref{eq: inequality for Em, gm} we may write
\begin{align*}		
		\abs{E_m(x,y)}&\le C_1 \frac{(m+2)^{2+\nu}}{m!}(\delta a(x,y))^m\le C_2 \frac{m^{2+\nu}}{m!}(\delta a(x,y))^m
	\end{align*}for all $m\ge \nu+1$, for some positive constants  $C_1,C_2$ depending only on $n,k$ and $\nu$. To complete the proof it suffices to choose a positive constants $C_3$ such that 
	$$\frac{e^2}{2} \frac{(m+2)^2}{\abs{(1+\gm)_m}}\le C_3 \frac{m^{2+\nu}}{m!}$$ for all $m=1,\dots,\nu+1$.
\end{proof}
\begin{corollary}\label{cor: estimates of Ek}
	Assume that $\mr{Re}(\gm)>-\nu$ where $\nu$ is a positive integer. Then there exists a constant $\delta\ge 1$ and  a positive constant $C:=C(n,k,\nu)$ such that 
	$$\abs{E_k(x,y)}\le C (\delta a(x,y)+1)^{\n+2} e^{\delta a(x,y)},$$ for all $x\in\R^2$ and $y\in\C^2$.
\end{corollary}
\begin{proof}
Using the fact that $E_{k}(x,y)=\sum_{m=0}^{\infty}E_m(x,y)$  the result of Theorem \ref{thm: estimates of Em}, we see that there exists $\delta\ge 1$ and a positive constant $C:=C(n,k,\nu)$ such that
\begin{align*}
\abs{E_k(x,y)}&\le C \sum_{m=0}^{\infty}\frac{(m+1)^{\n+2}}{m!}(\delta a(x,y))^m,
\end{align*} for all $x\in\R^2$, $y\in\C^2$. 
Noting that  $$(m+1)^{\nu+2}\le  (m+\nu+2)(m+\nu+1)\times\dots\times (m+1),$$ for all $m\in\Zp$,  one gets for some positive constant $C$,
\begin{align*}
\abs{E_k(x,y)}&\le C \abs{\frac{\pd^{\nu+2}}{\pd z^{\nu+2}}(z^{\nu+2}e^{z})}_{|z=\delta a(x,y)},
\end{align*}hence for some constant $C$, that changes from line to line, we have 
\begin{align*}
\abs{E_k(x,y)}&\le C \left((\delta a(x,y))^{\nu+2}+\dots+\delta a(x,y)+1\right)e^{\delta a(x,y)},
\end{align*} for all $x\in\R^2$,  $y\in\C^2$, from which the proof follows.
\end{proof}

\bibliographystyle{abbrv}
\bibliography{../Biblio/DunklOp-biblio}

\end{document}